\documentclass[12pt]{amsart}

\usepackage{a4,amsmath,amssymb,url}
\usepackage{showlabels}
\usepackage[colorlinks=true,urlcolor=blue,linkcolor=blue,citecolor=blue]{hyperref}

\numberwithin{equation}{section}
\theoremstyle{plain}
\newtheorem{thm}{Theorem}[section]
\newtheorem{lem}[thm]{Lemma}

\theoremstyle{definition}

\newtheorem{que}[thm]{Question}

\theoremstyle{remark}

\newcommand{\A}{{\mathbf A}}

\newcommand{\F}{{\mathbf F}}

\renewcommand{\L}{\mathbf{L}}

\newcommand{\N}{{\mathbb{N}}}
\newcommand{\Z}{{\mathbb{Z}}}

\newcommand{\SMP}{\mathrm{SMP}}

\newcommand{\Clo}{\mathrm{Clo}}

\newcommand{\algop}[2]{\langle {#1}, {#2} \rangle}

\newcommand{\setsuchthat}{\ : \ }%{\,\pmb{|}\,}  
\newlength{\ldprobleft} 
\newlength{\ldprobmid}  
\newlength{\ldprobright}

\newcommand{\dproblem}[3]{
\begin{equation*}
\parbox{\textwidth}{
\begin{tabular}{ @{} p{\ldprobleft} p{\ldprobmid} p{\ldprobright} @{} }
& \multicolumn{2}{l}{\textbf{\uppercase{#1}}} \\
& {\begin{minipage}[t]{\ldprobmid}Input:\vspace{1.5pt}\end{minipage}} & {\begin{minipage}[t]{\ldprobright}#2\vspace{1.5pt}\end{minipage}} \\ 
& {\begin{minipage}[t]{\ldprobmid}Problem:\end{minipage}} & {\begin{minipage}[t]{\ldprobright}#3\end{minipage}} \\
\end{tabular}}
\end{equation*}
}

\date{March 2, 2022}

  
\begin{document}

%%%%%%%%%%%%%%%%%%%%%%%%%%%%%%%%%%%%%%%%%%%%%%%%%%%%%%%%%%%%%%%%%%%
%%  FRONT MATTER                                                 %%
%%%%%%%%%%%%%%%%%%%%%%%%%%%%%%%%%%%%%%%%%%%%%%%%%%%%%%%%%%%%%%%%%%%

\title{Vaughan--Lee's nilpotent loop of size $12$ is finitely based}

\author{Peter Mayr}
\address[Peter Mayr]{Department of Mathematics,
University of Colorado,
Boulder, Colorado, USA}
\email{peter.mayr@colorado.edu}

\thanks{This material is based upon work supported by
%the Austrian Science Fund (FWF) grant no.\ P24285 and
the National Science Foundation grant no.\ DMS 1500254}

\begin{abstract}
 From work of Vaughan--Lee in~\cite{Va:NPV} it follows that if a finite nilpotent loop splits
 into a direct product of factors of prime power order, then its equational theory has a finite basis.
 Whether the condition on the direct decomposition is necessary has remained open since.  
 In the same paper, Vaughan--Lee gives an explicit example of a nilpotent loop of order $12$
 that does not factor into loops of prime power order and asks whether it is finitely based.

 We give a finite basis for his example by explicitly characterizing its term functions.
 This also allows us to show that the subpower membership problem for this loop can be solved in
 polynomial time.  
\end{abstract}

\maketitle

\section{Introduction}

 In~\cite{Va:NPV} Vaughan--Lee showed that every finite nilpotent Mal'cev algebra of finite type that has an
 equationally defined constant and splits into a direct product of algebras or prime power order
 is finitely based. In~\cite{FM:CTC} Freese and McKenzie generalized this result by removing the assumption
 that the algebra has a constant.
 Hence a finite nilpotent Mal'cev algebra $\A$ has a finite basis for its equational theory if it satisfies
 one of the following equivalent conditions for some $n\in\N$:
\begin{enumerate}
\item
 $\A$ is a direct product of algebras of prime power order;
\item
 all \emph{commutator terms} on $\A$, i.e., term functions satisfying $c(x_1,\dots,x_k,z) = z$ whenever
 $x_i=z$ for some $i\leq k$, have essential arity at most $n+1$;
\item
  $\A$ is $n$-\emph{supernilpotent}, i.e., $[1_A,\dots,1_A] = 0_A$ for the $(n+1)$-ary \emph{higher commutator}.
\end{enumerate}
 The equivalence (1)$\Leftrightarrow$(2) follows from~\cite[Theorem 3.14]{Ke:CMVS} and (1)$\Leftrightarrow$(3)
 from~\cite[Lemma 7.6]{AM:SAHC}.

 Vaughan--Lee, Freese and McKenzie proved their finite basis results syntactically by developing normal forms
 for terms with respect to commutator terms and then making crucial use of the bounded arity of the latter.
 Since their work in the 1980s, it remains open whether conditions (1)-(3) above are necessary, that is:

\begin{que} \label{qu:basis}
 Is every finite nilpotent Mal'cev algebra of finite type finitely based?  
\end{que}   

 Unlike for groups, not every finite nilpotent Mal'cev algebra is isomorphic to a direct product of
 algebras of prime power order. Already in~\cite{Va:NPV} Vaughan--Lee gave the following example of an
 indecomposable $12$-element nilpotent loop $\L$ and asked whether it is finitely based.

 For $x\in\Z_{12}$ and $C := 4\Z_{12}$, let $\bar{x} := x+C$.
 Define binary operations
\[ t\colon \Z_{12}^2\to C, (x,y)\mapsto\begin{cases} 4 & \text{if } (x,y)\in \bar{1}\times\bar{3} \cup\bar{3}\times\bar{1}, \\ 0 & \text{else}, \end{cases} \]
 and
\[ x\cdot y := x+y+t(x,y) \quad \text{ for } x,y\in\Z_{12}. \] 
 Then $\L := \algop{\Z_{12}}{\cdot}$ is a commutative $2$-nilpotent loop with center $C$.
 The congruences of $\L$ are induced by the normal subloops $0,C,2\Z_{12},\Z_{12}$.
 In particular $\L$ is subdirectly irreducible and does not split into a direct product of loops of prime power orders.
 Consequently it is not supernilpotent.

 To the best of our knowledge there are no published results on the equational bases for any nilpotent,
 not supernilpotent Mal'cev algebras.
 We present a detailed analysis of Vaughan--Lee's loop $\L$ as a test case and starting point for a future,
 hopefully more systematic approach to Question~\ref{qu:basis}.

 First we note that $\L$ has the group $\algop{\Z_{12}}{+}$ as reduct, and
 we explicitly characterize its set of $k$-ary term functions $\Clo_k(\L)$ for any $k\in\N$.

\begin{thm} \label{th:clone}
 For $k\in\N$ let 
\[ W_k := \{ w\colon\Z_{12}^k\to C \setsuchthat w(2\Z_{12}^k) = 0, w(x) = w(-x) \text{ for all } x\in\Z_{12}^k \}. \]
 Then $\Clo_k(\L) = \Clo_k(\algop{\Z_{12}}{+}) \oplus W_k$.
\end{thm}

 Here $\oplus$ denotes the direct sum in the group of all $k$-ary functions $\algop{\Z_{12}}{+}^{\Z_{12}^k}$.
 The proof of Theorem~\ref{th:clone} makes up all of Section~\ref{se:clone}.

 From Theorem~\ref{th:clone} we obtain normal forms for term functions and terms on $\L$ and eventually
 a finite basis in Section~\ref{se:basis}. 
 
\begin{thm} \label{th:basis}
 $\L$ is finitely based.
\end{thm}

 A description of the term clone as above also has computational applications.
 The \emph{Subpower Membership Problem} $\SMP$ for an algebra $\A$ is the following decision problem:
\dproblem{$\SMP(\A)$}{
$a_1,\ldots,a_k,b \in A^n$}{
 Is $b$ in the subalgebra $\langle a_1,\ldots,a_k\rangle$ of $\A^n$ generated by $a_1,\ldots,a_k$?}

 Willard observed that the Subpower Membership Problem for finite groups and rings can be solved in polynomial time
 and asked whether it is in P for any finite algebra with cube term (in particular, for every Mal'cev algebra)
 ~\cite{IMMVW}. This question remains open in general but there are some partial results.
 Together with Bulatov and Szendrei, the author showed that $\SMP(\A)$ is in NP for every finite algebra
 $\A$ with a cube term~\cite{BMS:SMP}.
 In~\cite{Ma:SMP} we gave a polynomial time algorithm for $\SMP$ for finite supernilpotent Mal'cev algebras
 using commutator terms and the techniques developed by Vaughan--Lee, Freese and McKenzie for their finite basis results.
 Like the answer for the finite basis question, we do not know the computational complexity of subpower membership
 for nilpotent, not supernilpotent Mal'cev algebras (even for loops) in general.

 However, using the explicit description of term functions in Theorem~\ref{th:clone}, it is not hard to reduce $\SMP$ for
 Vaughan--Lee's loop $\L$ to $\SMP$ of the group $\algop{\Z_{12}}{+}$. Hence we obtain the following in Section~\ref{se:smp}:

\begin{thm} \label{th:smp}
 $\SMP(\L)$ is in $P$.
\end{thm}

 We mention without proof that similar but somewhat less involved arguments show that all nilpotent loops of
 order $6$ or $10$ are finitely based and have tractable $\SMP$ as well.

 Our approach in this paper is tailored to the particular example we consider and has obvious limitations.
 In general a full characterization of term functions of arbitrary nilpotent Mal'cev algebras as in
 Theorem~\ref{th:clone} may not be feasible even when restricting the nilpotence degree.
 Furthermore, a description of the term clone does not automatically yield a finite basis for the algebra
 as we see from the proof of Theorem~\ref{th:basis}.
 
 Still our admittedly limited evidence leads us to
 propose the following problems as steps towards a possible positive solution of Question~\ref{qu:basis}.

\begin{que} 
 Does every finite nilpotent Mal'cev algebra have a supernilpotent reduct?
\end{que}  

\begin{que} 
 Is every nilpotent expansion of a finite abelian (more generally supernilpotent) Mal'cev algebra finitely based?
\end{que}

 For $n\in\N$, let $[n]:= \{0,\dots,n-1\}$.
 
\section{The clone of term operations (proof of Theorem~\ref{th:clone})} \label{se:clone}

 We split the proof of Theorem~\ref{th:clone} into several lemmas. First we show that $\L$
 has a group reduct.

\begin{lem} \label{le:t}
 $\L$ is term equivalent to $\algop{\Z_{12}}{+,t}$.
\end{lem}

\begin{proof}
 Let $x,y\in\Z_{12}$. Note that $x^2 = 2x$ and consequently
%\[ (xy)^2 = 2x+2y+2t(x,y). \]
\[ (xy)^4 = 4x+4y+t(x,y). \]
 Further $\cdot$ and $+$ are equal on $2\Z_{12}$ which yields
%\[ (xy)^2 \cdot (x^2)^5 \cdot (y^2)^5 = 2t(x,y). \]
\[ (xy)^4 \cdot x^8 \cdot y^8 = t(x,y). \]
 Hence $t$ is a term operation on $\L$ and the assertion follows.
\end{proof}

 Note that $\Clo_k(\algop{\Z_{12}}{+})$ and $W_k$ are subgroups of $\Z_{12}^{\Z_{12}^k}$ by their definitions. We claim that
\begin{equation} \label{eq:direct}
 \text{ the sum } \Clo_k(\algop{\Z_{12}}{+}) + W_k \text{ is direct}.
\end{equation} 
 Let $w\in \Clo_k(\algop{\Z_{12}}{+}) \cap W_k$. Then $w(x) = \sum_{i=1}^k a_ix_i$ for some $a_1,\dots,a_k\in\Z$.
 Now $w(\Z_{12}^k) \subseteq C$ and $w(2\Z_{12}^k) = 0$  imply that all coefficients $a_1,\dots,a_k$ are multiples of $4$
 and of $6$, hence of $12$. Thus $w=0$ and~\eqref{eq:direct} is proved.

 For
\begin{equation} \label{eq:<}
  \Clo_k(\L) \subseteq \Clo_k(\algop{\Z_{12}}{+}) + W_k,
\end{equation}  
 it suffices that $\Clo_k(\algop{\Z_{12}}{+})+ W_k$ is closed under $+$ and $t$.
 Only closure under $t$ remains to be shown. Note that for any $a_1,\dots, a_k,b_1,\dots,b_k\in\Z$ and $w,v\in W_k$, 
\[ g(x) := t(\sum_{i=1}^k a_i x_i + w(x), \sum_{i=1}^k b_i x_i + v(x)) = t(\sum_{i=1}^k a_i x_i, \sum_{i=1}^k b_i x_i) \] 
 satisfies $g(2\Z_{12}^k) = 0$ and $g(x) = g(-x)$ for all $x\in\Z_{12}^k$ by the definition of $t$.
 Hence $g\in W_k$ and~\eqref{eq:<} follows.

 The converse containment $\supseteq$ in Theorem~\ref{th:clone} will follow from Lemma~\ref{le:clone}
 after establishing some more auxiliary results.
 First note that $W_k$ forms a vector space over the $3$-element field $\F_3$ of dimension $(4^k-2^k)/2$.
 To obtain normal forms for the functions in $W_k$ we construct a basis of term functions.
 
\begin{lem} \label{le:gk}
 Let $k\in\N$, and
 \[ g_k\colon \Z_{12}^k\to C,\ x\mapsto\begin{cases} 4 & \text{if } x_1\in 1+2\Z_{12}, x_2,\dots,x_k\in C, \\ 0 & \text{else}. \end{cases} \]
 Then $g_k$ is a term operation of $\L$.
\end{lem}

\begin{proof}
 We use induction on $k$. Note that $g_1(x) = t(x,-x)$ and $g_2(x,y) = t(x,3x+y)$ are term operations of $\L$ by
 Lemma~\ref{le:t}. So assume $k\geq 3$ in the following.
 By induction assumption we have the term function $h := g_{k-1}$. We claim that
\begin{align*} g_k(x_1,\dots,x_k) = & \sum_{b\in [4]} h(x_1,x_2+bx_3, x_4,\dots,x_k) \\
 & + \sum_{a\in\{0,2\}} h(x_1,ax_2+x_3,x_4\dots,x_k) \\
 & +\sum_{(a,b)\in\{0,2\}^2} h(x_1,ax_2+bx_3,x_4,\dots,x_k).
\end{align*} 
 By definition the functions on either side of the equation are $0$ whenever $x_1\in 2\Z_{12}$ or $x_4,\dots,x_k\not\in C$.
 Further they are constant for $x_1\in 1+2\Z_{12}$ and overall constant on $C$-blocks. 
 So it suffices to consider the graphs for the functions on the right hand side of the equation for
 $x_1=1$ and $x_4=\dots=x_k=0$. We give the second argument $ax_2+bx_3$ and the table for $h(1,ax_2+bx_3,0,\dots,0)$
 on $C$-blocks where $x_2$ runs through the representatives $0,2,1,3$ left to right and $x_3$ runs through
 $0,2,1,3$ bottom to top.
\[ \begin{array}{cccccc}
 x_2+0x_3 & x_2+1x_3 & x_2+2x_3 & x_2+3x_3 & 0x_2+x_3 & 2x_2+x_3 \\ 
& \\
 \begin{smallmatrix} 4&0&0&0\\4&0&0&0\\4&0&0&0\\4&0&0&0\\ \end{smallmatrix} &
 \begin{smallmatrix} 0&0&4&0\\0&0&0&4\\0&4&0&0\\4&0&0&0\\ \end{smallmatrix} &
 \begin{smallmatrix} 0&4&0&0\\0&4&0&0\\4&0&0&0\\4&0&0&0\\ \end{smallmatrix} &
 \begin{smallmatrix} 0&0&0&4\\0&0&4&0\\0&4&0&0\\4&0&0&0\\ \end{smallmatrix} &
 \begin{smallmatrix} 0&0&0&0\\0&0&0&0\\0&0&0&0\\4&4&4&4\\ \end{smallmatrix} &
 \begin{smallmatrix} 0&0&0&0\\0&0&0&0\\0&0&4&4\\4&4&0&0\\ \end{smallmatrix} 
\end{array} \]

\[ \begin{array}{cccc}
 0x_2+0x_3 & 0x_2+2x_3 & 2x_2+0x_3 & 2x_2+2x_3 \\
& \\
 \begin{smallmatrix} 4&4&4&4\\4&4&4&4\\4&4&4&4\\4&4&4&4\\ \end{smallmatrix} &
 \begin{smallmatrix} 0&0&0&0\\0&0&0&0\\4&4&4&4\\4&4&4&4\\ \end{smallmatrix} &
 \begin{smallmatrix} 4&4&0&0\\4&4&0&0\\4&4&0&0\\4&4&0&0\\ \end{smallmatrix} &
 \begin{smallmatrix} 0&0&4&4\\0&0&4&4\\4&4&0&0\\4&4&0&0\\ \end{smallmatrix} 
\end{array} \]
 Adding the tables in the first row yields
\[ \begin{smallmatrix} 4&4&4&4\\4&4&4&4\\8&8&4&4\\0&8&4&4\\ \end{smallmatrix} \]
 and adding those in the second row yields
\[ \begin{smallmatrix} 8&8&8&8\\8&8&8&8\\4&4&8&8\\4&4&8&8\\\end{smallmatrix} \]
 Their sum is the table of $g_k(1,x_2,x_3,0,\dots,0)$,
\[ \begin{smallmatrix} 0&0&0&0\\0&0&0&0\\0&0&0&0\\4&0&0&0\\ \end{smallmatrix} \]
 This proves that $g_k$ is a term function.
\end{proof}

 To simplify the description of term functions on $\L$ we show that $\L$ is term equivalent to the expansion of the
 group $\algop{\Z_{12}}{+}$ with the binary function $g_2$ from the previous lemma. We mention without proof that $\L$ is not
 equivalent to an expansion of $\algop{\Z_{12}}{+}$ with unary functions.

\begin{lem} \label{le:f}
 Let $f:= g_2$. Then $\L$ is term equivalent to $\algop{\Z_{12}}{+,f}$.
\end{lem}

\begin{proof}
 This follows from $f(x,y) = t(x,3x+y)$ and $t(x,y) = f(x,x+y)$ for all $x,y\in\Z_{12}$.
\end{proof}

 Finally we define a simple basis of $W_k$.

\begin{lem} \label{le:clone}
 Let $k\in\N$, and let $R_k$ contain exactly one generator for every cyclic subgroup of order $4$ of $\algop{\Z_{12}/C}{+}^k$.
 For $\bar{r}\in R_k$, define
\[ f_{\bar{r}}\colon\Z_{12}^k\to C,\ x\mapsto \begin{cases} 4 & \text{if } x\in \bar{r} \cup-\bar{r}, \\ 0 & \text{else}. \end{cases} \]
 Then
\begin{enumerate}
\item \label{it:fr}
 $f_{\bar{r}}\in\Clo_k(\L)$ and %is a term function on $\L$.
\item \label{it:Wk}
 $\{ f_{\bar{r}} \setsuchthat \bar{r}\in R_k\}$ is a basis for $W_k$ as a vector space over $\F_3$.
 In particular $W_k\subseteq \Clo_k(\L)$.
\end{enumerate}
\end{lem}

\begin{proof}
 For~\eqref{it:fr} note that $f_{(\bar{1},\bar{0},\dots,\bar{0})} = g_k$ is a term function by Lemma~\ref{le:gk}.
 For every $\bar{r}\in R_k$ there exists an automorphism $\bar{\varphi}$ on the group $(\Z_{12}/C)^k$ that maps $\bar{r}$ to
 $(\bar{1},\bar{0},\dots,\bar{0})$. There exist $k$-ary term functions $\varphi_1,\dots,\varphi_k$ on $\algop{\Z_{12}}{+}$
 such that $\varphi(x) := (\varphi_1(x),\dots,\varphi_k(x))$ induces $\bar{\varphi}$.
 Hence $f_{\bar{r}}(x) = f_{(\bar{1},\bar{0},\dots,\bar{0})}(\varphi_1(x),\dots,\varphi_k(x))$ is a term function as well.  

 For~\eqref{it:Wk} note that for distinct $\bar{r},\bar{s}\in R_k$ the supports of $f_{\bar{r}}$ and $f_{\bar{s}}$ are disjoint.
 Hence $B := \{ f_{\bar{r}} \setsuchthat \bar{r}\in R_k \}$ is linearly independent.
 Further $W_k$ consists exactly of the functions that are constant on the supports of $f_{\bar{r}}$ for $\bar{r}\in R_k$.
 Thus $B$ spans $W_k$ and is a basis. 
\end{proof}

 By Lemma~\ref{le:clone} we have that $\Clo(\algop{\Z_{12},+})+W_k$ is contained in $\Clo(\L)$,
 which concludes the proof of Theorem~\ref{th:clone}.
 Further every $h\in\Clo_k(\L)$ has a unique normal form
\begin{equation} \label{eq:normalform}
  h(x) = \sum_{i=1}^k a_i x_i + \sum_{\bar{r}\in R_k} b_{\bar{r}} f_{\bar{r}}
\end{equation}
for coefficients $a_i\in [12], b_{\bar{r}}\in [3]$.

\section{Subpower membership (Proof of Theorem~\ref{th:smp})} \label{se:smp}

 Let $a_1,\dots,a_k,b\in\Z_{12}^n$ be the input of $\SMP(\L)$. Define an equivalence relation $\sim$ on
 $\{1,\dots,n\}$ by 
\[ i\sim j \text{ if } (a_{1i}, \dots, a_{ki}) \equiv \pm (a_{1j}, \dots, a_{kj}) \bmod C^k. \]
% Here we consider two $k$-tuples as congruent modulo $C$ if all their components are.
 Equivalently $i\sim j$ iff the classes of $(a_{1i}, \dots, a_{ki})$ and  of $(a_{1j}, \dots, a_{kj})$ in $(\Z_{12}/C)^k$
 generate the same cyclic subgroup.

 For $i\leq n$ with $(a_{1i}, \dots, a_{ki}) \not\in 2\Z_{12}^k$ define
\[ b_i\in\Z_{12}^n \text{ by } b_{ij} := \begin{cases} 4 & \text{if } i\sim j, \\ 0 & \text{else.} \end{cases} \]
 Then $b_i = f_{\bar{r}}(a_1,\dots,a_k)$ for $\bar{r}$ the coset of $(a_{1i}, \dots, a_{ki})$ modulo $C$ and $f_{\bar{r}}$
 as defined in Lemma~\ref{le:clone}.
 Let $\{ b_{i_1},\dots,b_{i_\ell}\}$ be the set of tuples obtained in this way. Note that $\ell\leq n$.
 By Theorem~\ref{th:clone} and the normal form given in~\eqref{eq:normalform} we see that
 the subloop of $\L^n$ that is generated by $a_1,\dots,a_k$ is equal to the subgroup of $\algop{\Z_{12}}{+}^n$ generated
 by $a_1,\dots,a_k,b_{i_1},\dots,b_{i_\ell}$. 
 Since membership for the latter can be checked in polynomial time in $(k+\ell)n$, that is, in $kn$, by linear algebra,
 so can membership for the former.
 This proves Theorem~\ref{th:smp}.

\section{Equational basis (Proof of Theorem~\ref{th:basis})} \label{se:basis}

 Since $\L$ and $\A:=\algop{\Z_{12}}{+,-,0,f}$ are term equivalent by Lemma~\ref{le:f},
 it suffices to show that the latter is finitely based.
 We claim that the standard identities for the abelian group $\algop{\Z_{12}}{+,-,0}$ together with
 the following form a finite basis for the equational theory of $\A$:
\begin{enumerate}
\item \label{it:ff}
 $f(x+f(u,v), y +f(r,s)) = f(x,y)$
\item \label{it:x+2u,y+4v}
 $f(x+2u,y+4v) = f(x,y)$
\item \label{it:3f}
 $3f(x,y) = 0$
\item \label{it:0,y} 
 $f(0,y) = 0$
\item \label{it:x,3y}
 $f(x,3y) = f(x,y)$
\item \label{it:x+y,y}
  $f(x+y,y) = f(x,y)$
\item \label{it:x+y,z}
 $f(x+y,z) = f(x,2x+2y+2z)+f(x,2y+z)-f(x,2y)+f(x,2z)+f(y,z)$
\item \label{it:x,2x+y}
 $f(x,2x+y) =  f(x,2y)-f(x,y)$
%\item \label{it:y,x+3z}
% $f(y,x+3z) = f(y,x+z)-f(x,x+z)+f(x,x+2y+z)$
% $f(x,3y) = f(x,y)$
\item \label{it:x,2y+z}
 $f(x,2y+z) = f(x,z)-f(y,z)+f(y,2x+z)$
%\item $f(x,x+3y) = f(x,0)-f(x,2y)-f(x,x+y)$
\end{enumerate}
 Identities~\eqref{it:ff} to~\eqref{it:x+y,y} are immediate from the definition of
\[ f\colon\Z_{12}^2\to 4\Z_{12},\ (x,y) \mapsto \begin{cases} 4 & \text{if } x\in 1+2\Z_{12}, y\in 4\Z_{12}, \\ 0 & \text{else.} \end{cases} \]
% using that $f$ is constant modulo $D := 2\Z_{12}$ in the first argument, constant modulo $C$ in the second argument, etc.
 The remaining equations may appear less intuitive at first glance but they actually arise naturally when interpolating
 ternary functions in $W_3$ using the binary $f$ and linear maps on $\Z_{12}$. For showing their correctness,
 let $C = 4\Z_{12}$ and $D := 2\Z_{12}$. 

 For~\eqref{it:x,2x+y} note that both sides of the equation vanish if $x\in D$. So assume $x\in 1+D$.
 Then the left hand side $f(1,2+y)$ is $4$ if $y\in 2+C$ and $0$ else.
 This is equal to the right hand side $f(1,2y)-f(1,y)$.

 Next~\eqref{it:x,2y+z} is immediate by~\eqref{it:x+2u,y+4v} if $x\in D$ or $y\in D$.   
 So assume $x,y \in 1+D$. Then $2x,2y\in 2+C$ and~\eqref{it:x+2u,y+4v} yields $f(x,z) = f(y,z)$ as well as
 $f(x,2y+z) = f(y,2x+z)$. Hence~\eqref{it:x,2y+z} follows.

 Finally~\eqref{it:x+y,z} follows from~\eqref{it:x+2u,y+4v} if $x\in D$ because then all but the last summand on the
 right hand side vanish. If $y\in D$, then $2y\in C$ and by~\eqref{it:x+2u,y+4v} the right hand side reduces to
 \[ \underbrace{f(x,2x+2z)}_{=f(x,0)-f(x,2z) \text{ by}~\eqref{it:x,2x+y}}+f(x,z)-f(x,0)+f(x,2z)+\underbrace{f(0,z)}_{=0} = f(x,z), \]
 same as the left hand side.
 So assume $x,y \in 1+D$. Then $2x,2y\in 2+C$ and the right hand side simplifies to
 \[ f(x,2z)+\underbrace{f(x,2x+z)}_{=f(x,2z)-f(x,z) \text{ by}~\eqref{it:x,2x+y}}-\underbrace{f(x,2)}_{=0}+f(x,2z)+f(x,z) = 0. \]
 Since $x+y\in D$, the left hand side $f(x+y,z)$ yields $0$ as well. Thus~\eqref{it:x+y,z} is proved. 

 Next we show that using the identities of the group $\Z_{12}$ together with~\eqref{it:ff}-\eqref{it:x,2y+z}
 we can transform terms into normal forms.
 By~\eqref{it:ff} and~\eqref{it:x+2u,y+4v} every $k$-ary term in the language of $\A$ can be transformed
 into a sum of variables and terms 
\[ f(\sum_{i=1}^{k} v_i x_i,\sum_{j=1}^{k} w_j x_j) \text{ with } v_1,\dots,v_k\in [2], w_1,\dots,w_k\in [4]. \]
 The latter can be rewritten into a sum of
\[ f(x_i,\sum_{j=1}^{k} w_j x_j) \text{ for } i\leq k, w_1,\dots,w_k\in [4] \]
 by~\eqref{it:0,y} and ~\eqref{it:x+y,z}.
 By~\eqref{it:x,2x+y} we may assume that $w_i\in [2]$. To show that $w_j$ can be chosen to be in $[2]$ for all $j\leq i$
 above, consider the smallest $r$ such that $1\leq r <i$ and $w_r\in\{2,3\}$.
 For $z:= (w_r-2)x_r+ \sum_{j\neq r} w_j x_j$ equation~\eqref{it:x,2y+z} yields
\[ f(x_i,2x_r+z) = f(x_i,z) - f(x_r,z)+f(x_r, 2x_i+z). \]
 For all $j\leq r$ the coefficients of $x_j$ in $z$ are in $\{0,1\}$. Repeating this procedure for $f(x_i,z)$
 if necessary, eventually yields that every term of $\A$ can be written as sum of variables and  
\[ f(x_i,\sum_{j=1}^{i-1} a_j x_j + bx_i + \sum_{j=i+1}^{k} c_j x_j) \text{ for }  i\leq k, a_1,\dots,a_{i-1},b\in [2], c_{i+1},\dots,c_k\in [4]. \]
 We can give further restrictions on the coefficients $a_j,c_j$ above depending on whether $b$ is $0$ or $1$.

 {\bf Case $b=0$:} Note that
\begin{align*}
 f(x,y+3z) & = f(x,3y+z) \text{ by }\eqref{it:x,3y} \\
 & = f(x,2y+\underline{y+z}) \\
 & = f(x,\underline{y+z})-f(y,\underline{y+z})+f(y,2x+\underline{y+z}) \text{ by }\eqref{it:x,2y+z}.
\end{align*}
 Let $c := (c_{i+1},\dots,c_k)\in [4]^{k-i}\setminus \{0,2\}^{k-i}$. Then $3c\not\equiv c \bmod 4$. 
 Using the identity above %~\eqref{it:y,x+3z}
 for $x=x_i, y = \sum_{j=1}^{i-1} a_j x_j, z =\sum_{j=i+1}^{k} c_j x_j$ together with~\eqref{it:x+y,z} we have that
\[ f(x_i,\sum_{j=1}^{i-1} a_j x_j + 3 \sum_{j=i+1}^{k} c_j x_j) \] 
 can be expressed as a sum of $f(x_i,\sum_{j=1}^{i-1} a_j x_j + \sum_{j=i+1}^{k} c_j x_j)$ and terms $f(x_j,\dots)$ for
 $j<i$.

 {\bf Case $b=1$:} Note that
\begin{align*}
 f(x,x+y+3z) = & f(x,3x+3y+z) \text{ by }\eqref{it:x,3y} \\
 = & f(x,2x+\underline{x+3y+z}) \\
 = & f(x,2x+\underline{2y+2z})-f(x,x+3y+z) \text{ by }\eqref{it:x,2x+y} \\
 = & [f(x,0)-f(x,2y+\underline{2z})]-f(x,2y+\underline{x+y+z}) \text{ by }\eqref{it:x,2x+y} \\
 = & f(x,0)- [f(x,2z)-f(y,2z)+f(y,2x+2z)] \\ 
   & - [f(x,x+y+z)-f(y,x+y+z)+f(y,3x+y+z)]\text{ by }\eqref{it:x,2y+z}.
\end{align*}
 Let $c := (c_{i+1},\dots,c_k)\in[4]^{k-i}\setminus \{0,2\}^{k-i}$. Then $3c\not\equiv c \bmod 4$. 
 Using the identity above for $x=x_i, y = \sum_{j=1}^{i-1} a_j x_j, z =\sum_{j=i+1}^{k} c_j x_j$ together
 with~\eqref{it:x+y,z} we have that
 \[ f(x_i,\sum_{j=1}^{i-1} a_j x_j + x_i+ 3 \sum_{j=i+1}^{k} c_j x_j) \]
 can be expressed as a sum of $f(x_i,0), f(x_i,2\sum_{j=i+1}^{k} c_j x_j),
 f(x_i,\sum_{j=1}^{i-1} a_j x_j + x_i + \sum_{j=i+1}^{k} c_j x_j)$ and terms $f(x_j,\dots)$ for $j<i$.

 If $c_{i+1},\dots,c_k\in\{0,2\}$, then by~\eqref{it:x+y,y} and~\eqref{it:x+2u,y+4v}
\begin{align*}
  f(x_i,\sum_{j=1}^{i-1} a_j x_j + x_i + \sum_{j=i+1}^{k} c_j x_j)
& = f(\sum_{j=1}^{i-1} a_j x_j, \sum_{j=1}^{i-1} a_j x_j + x_i + \sum_{j=i+1}^{k} c_j x_j)
\end{align*}
 with the latter expressible as a sum of terms $f(x_j,\dots)$ for $j<i$ by~\eqref{it:x+y,z}.

 For $\ell\in\N$ define an equivalence relation $\sim$ on $[4]^\ell \setminus \{0,2\}^\ell$ by $c\sim d$ if
 $c \equiv \pm d \bmod 4$.
 Let $R_\ell$ be a transversal for $\sim$. Note that this is essentially the same set as defined in Lemma~\ref{le:clone}.
 Again $|R_\ell| = (4^\ell-2^\ell)/2$.

 In summary, using the identities of the group $\Z_{12}$ and ~\eqref{it:ff} to~\eqref{it:x,2y+z} as well as recursion on
 $i\in [k]$, we managed to rewrite any $k$-ary term $h$ of $\A$ as a sum of variables and terms
\begin{align*} 
 s_{a,c} & := f(x_i,\sum_{j=1}^{i-1} a_j x_j+ \sum_{j=i+1}^k c_j x_j)  \text{ for } i\in [k], a\in [2]^{i-1}, c\in R_{k-i}\cup \{0,2\}^{k-i}, \\
 t_{a,c} & := f(x_i,\sum_{j=1}^{i-1} a_j x_j+x_i+ \sum_{j=i+1}^k c_j x_j)  \text{ for } i\in [k], a\in [2]^{i-1}, c\in R_{k-i}.
\end{align*}
 We claim that this representation is a normal form for $h$, more precisely,
 for $u_i\in [12], v_{a,c}, w_{a,c} \in [3]$ we have
\[ \sum_i u_i x_i + \sum_{a,c} v_{a,c} s_{a,c} + \sum_{a,c} w_{a,c} t_{a,c} = 0 \text{ iff  all } u_i, v_{a,c}, w_{a,c} \text{ are } 0. \]
 Assume that the term on the lefthand side induces the constant $0$-function on $\A$.
 From the direct decomposition in Theorem~\ref{th:clone} we see $u_1 = \dots = u_k = 0$.
 The functions induced by $s_{a,c},t_{a,c}$ span $W_k$ by their construction and by Lemma~\ref{le:clone}.
 To see that this spanning set for $W_k$ actually forms a basis it suffices to determine its size.
 For fixed $i\in [k]$ the number of choices for $a,c$ in $s_{a,c}$ and in $t_{a,c}$ together is
\[ 2^{i-1} [(4^{k-i}-2^{k-i})/2+2^{k-i}] + 2^{i-1} [(4^{k-i}-2^{k-i})/2] = 2^{i-1}\cdot 4^{k-i} = 2^{k-1}\cdot 2^{k-i}. \] 
 Hence summing up over all $i\in [k]$ we have
\[ 2^{k-1}\cdot\sum_{i=1}^k 2^{k-i} = 2^{k-1}(2^k-1) = \dim W_k \]
 functions. Hence they are linearly independent over $\F_3$, and $\sum_{a,c} v_{a,c} s_{a,c} + \sum_{a,c} w_{a,c} t_{a,c} = 0$
 only if all coefficients $v_{a,c}, w_{a,c}$ are $0$.

 Thus every $k$-ary term function on $\A$ can be rewritten into a uniquely determined normal form using the finitely
 many identities given above, and $\A$ is finitely based.
 Theorem~\ref{th:basis} is proved.

%\bibliography{../../biblio}
\bibliographystyle{plain}

\def\cprime{$'$}

\end{document}